\documentclass[11pt]{amsart}

\usepackage{mathrsfs,graphicx,latexsym,tikz,color,euscript}
\usepackage[T1]{fontenc}
\usepackage{lmodern,mathtools,microtype,amsfonts,amssymb,amsmath,amscd,amsthm}
\usepackage{epstopdf}
\usepackage{hyperref}
\usepackage{upgreek}
\usepackage{comment}
\usepackage{pdfsync}
\usepackage{enumitem}

\newtheorem{theorem}{Theorem}[section]
\newtheorem{lemma}[theorem]{Lemma}
\newtheorem{proposition}[theorem]{Proposition}

\theoremstyle{definition}
\newtheorem{definition}[theorem]{Definition}

\theoremstyle{remark}
\newtheorem{remark}[theorem]{Remark}

\numberwithin{equation}{section}

\newcommand{\rr}{\mathbb R}

\newcommand{\be}{\mathbb E}
\newcommand{\scc}{\mathscr C}

\newcommand{\ep}{\varepsilon}
\newcommand{\Om}{\Omega}
\newcommand{\dl}{\delta}

\newcommand{\lmd}{\lambda}

\newcommand{\sg}{\sigma}

\newcommand{\bn}{\mathbf n}
\newcommand{\bk}{\mathbf k}

\newcommand{\se}{\frac{2}{3}}
\newcommand{\ey}{\frac{1}{2}}

\newcommand{\ip}[2]{\langle #1,#2\rangle_m}
\newcommand{\norm}[1]{\lVert #1\rVert_m}
\newcommand{\rnorm}[1]{\lVert #1\rVert}
\newcommand{\dist}{\operatorname{dist}}
\newcommand{\Hess}{\operatorname{Hess}}

\begin{document}

\title[no infinite spin]{No infinite spin in the $n$-body problem in all dimensions}
 

\author{Zhe Wang}
\address{Chern Institute of Mathematics and LPMC, Nankai University, Tianjin, China}
\email{zhewang@mail.nankai.edu.cn}

\author{Guowei Yu}
\address{Chern Institute of Mathematics and LPMC, Nankai University, Tianjin, China}
\email{yugw@nankai.edu.cn}

\thanks{This work is supported by NSFC (No. 12671224), Nankai Zhide Fundation and the Fundamental Research Funds for the Central Universities.}

\begin{abstract} 

In the planar $n$-body problem, Moeckel and Montgomery \cite{MM25} showed that there is no infinite spin for total collision solutions assuming the reduced and normalized configuration converges to an isolated central configuration. Following their approach, generalizations have been obtained for partial collision solutions and parabolic solutions in the planar case, and for total collision solutions in the spatial case by various authors, all under the same assumption that the reduced and normalized configuration converges to an isolated central configuration. 

Here we prove there is no infinite spin for any collision or parabolic solution in all dimensions without such an assumption. 
\end{abstract}

\maketitle

\section{Introduction and main results}
The Newtonian $n$-body problem in $\rr^d$ with $d \ge 2$ studies the motion of $n$ positive point masses, $m_i$, $i \in \bn$ $=\{1, 2, \dots, n\}$, under Newton's universal gravitational law
\begin{equation} \label{eq;n-body}
 m_i\ddot{q_i}=\nabla_iU(q), \; \forall i \in \bn,
\end{equation}
where $q=(q_i)_{i \in \bn} \in\mathbb{R}^{dn}$ with each $q_i$ representing the location of $m_i$, and
$$  U(q)=\underset{i<j}{\sum}\frac{m_im_j}{r_{ij}}, \; r_{ij}=|q_i-q_j|. $$
Due to translation symmetry, we may fix the center of mass at the origin by assuming
$$ q \in \be :=\{ q \in \rr^{dn} : \sum_{i \in \bn} m_i q_i =0\}. $$

The global dynamics of the $n$-body problem are deeply affected by the set of collision configurations $\Delta$, which are singularities of \eqref{eq;n-body}
$$\Delta = \{q\in \rr^{dn}: q_i = q_j \text{ for some } i\neq j\}.$$
Hence it is important to understand collision solutions defined below. 
\begin{definition}\label{def;k-col}
	\label{df;par-coll} Given any $\bk \subset \bn$ with $|\bk| \ge 2$, we say that $q: [0, T) \to \be$ is a \textbf{$\bk$-collision} solution, if it is a collision-free solution for $0< t <T$, and $\lim_{t \to T} q(t) = q^* \in \Delta$ with 
	$$ \begin{cases}
		q^*_i = q^*_j, & \forall i \ne j \in \bk; \\
		q^*_i \ne q^*_j, & \forall i \in \bk, j \in \bk':= \bn \setminus \bk. 
	\end{cases}
     $$
\end{definition}
\begin{remark}
    A $\bk$-collision solution is partial, if $|\bk| < n$, and total, if $|\bk|= n$. There may be simultaneous partial collisions occurring at the same time. 
\end{remark}

Given any subset $\bk \subset \bn$, we denote the center of mass of the $\bk$-subsystem as
\begin{equation}
\label{eq;com} c_{\bk} = m^{-1}_{\bk} \sum_{i \in \bk} m_i q_i, \text{ where } m_{\bk} = \sum_{i \in \bk} m_i.
\end{equation}
To describe the relative motion of the bodies in the $\bk$-subsystem with respect to $c_{\bk}$, define the relative configuration as  
$$ q^c_{\bk}= (q_i - c_{\bk})_{i \in \bk} \in \be_{\bk}, \qquad \be_{\bk} :=\{ (q_i)_{i \in \bk}: \sum_{i \in \bk} m_i q_i =0\}.$$
The size of the subsystem then can be measured by 
\begin{align}\label{eq;mom-iner}
    I_{\bk}(q^c_{\bk})= \sum_{i \in \bk} m_i|q_i - c_{\bk}|^2
\end{align}
which is its moment of inertia with respect to $c_{\bk}$. Then the normalized relative configurations  
$$ \hat{q}^c_{\bk}  = q^c_{\bk}/ \sqrt{I_{\bk}(q^c_{\bk})} $$
form the following $d (|\bk|-1)-1$ dimensional sphere
$$ S_{\bk} := \{ q_{\bk} \in \be_{\bk}: I_{\bk}(q_{\bk})=1 \}.$$

Let $q(t)$ be a $\bk$-collision solution as in Definition \ref{df;par-coll} above. The energy of the $\bk$-subsystem then is defined as
\begin{align}\label{eq;sub-energy}
    h_\mathbf{k}(q, \dot{q}):= \underset{i\in\mathbf{k}}{\sum}\frac{m_i}{2}|\dot{q_i}- \dot{c}_{\bk}|^2 -U_\mathbf{k}(q^c_{\bk}), \quad U_\mathbf{k}(q^c_{\bk}) := \underset{i,j\in\mathbf{k},i<j}{\sum}\frac{m_im_j}{|q_i -q_j|}.
\end{align} 
By a classical result of Chazy \cite{Chazy18}, as $t$ approaches the collision moment $T$, the corresponding normalized relative configuration curve $\hat{q}^c_{\bk}(t)$ must converge to the set of normalized central configurations of the $\bk$-subsystem (for more details see Proposition \ref{prop;asym-col}). 
\begin{definition}
\label{df;cc} We say $p_{\bk} = (p_i)_{i \in \bk} \in \be_{\bk}$ is a central configuration or CC of the $\bk$-subsystem, if 
\begin{equation}
	\label{eq;CC}  \nabla_{i} U_{\bk}(p_{\bk}) + \lmd m_i p_i = 0, \; \forall i \in \bk
\end{equation}
for some constant $\lmd >0$. 
Denote the set of all CCs of the $\bk$-subsystem by $\mathscr{C}_{\bk}$. The subset of all normalized CCs then is  
$ \hat{\mathscr{C}}_{\bk} = \mathscr{C}_{\bk} \cap S_{\bk}$.
\end{definition}


If we view $\lmd$ in \eqref{eq;CC} as a Lagrange multiplier, then $\hat{\mathscr{C}}_{\bk}$ is precisely the set of critical points of the potential function $U_{\bk}$ restricted to $\tilde{S}_{\bk} = \{ s \in S_{\bk}: s_i \ne s_j \; \text{ for } i \ne j \}$, the set of all normalized collision-free configurations. 

By rotational symmetry, if $p_{\bk} \in \hat{\mathscr{C}}_{\bk}$ is a normalized CC, so is $g (p_{\bk})= (g(p_i))_{i \in \bk}$, for any rotation $g \in G = SO(d)$. Denote the quotient map $S_{\bk} \to S_{\bk}/SO(d)$ as $p_{\bk} \mapsto [p_{\bk}]$, and the corresponding potential function induced by $U_{\bk}$ as $[p_{\bk}] \mapsto U_{\bk}([p_{\bk}])$. We say $p_{\bk}$ is an isolated (resp. nondegenerate or degenerate) CC, if $[p_{\bk}]$ is an isolated (resp. nondegenerate or degenerate) critical point of $U_{\bk}([p_{\bk}])$. 

Recall that $q(t)$ is a $\bk$-collision solution, by the above explanation, the limit set of $[\hat{q}^c_{\bk}(t)]$ must be contained in $\hat{\scc}_{\bk}/SO(d)$, as $t \to T$. If we assume there is only one point $[p_{\bk}]$ in the limit set, then the limit set of $\hat{q}_{\bk}^c(t)$ is entirely contained in $G(p_{\bk}) = \{ g (p_{\bk}): g \in G \}$. When $d = 2$, $G(p_{\bk})$ is a circle. Then $\hat{q}_{\bk}^c(t)$ may converge to the whole circle and during the process undergoes infinitely many spins. Even though more complicated limiting behaviors are possible when $d \ge 3$, the following problem is usually referred to as \emph{the problem of infinite spin}.

\textbf{Problem 1}: \emph{Under the isolation assumption that $[\hat{q}^c_{\bk}(t)]$ converges to an isolated normalized CC when $t \to T$, is the limit set of $\hat{q}^c_{\bk}(t)$ a single point?}

\begin{remark}
	Although as we mentioned more complicated behaviors are possible, when $d \ge 3$, in this paper, by saying there is no infinite spin, we mean the limiting set of $\hat{q}^c_{\bk}(t)$ is a single point. 
\end{remark}

Historically there have been several failed attempts in the literature to solve the above problem. The first breakthrough is the work by Moeckel and Montgomery \cite{MM25}, where they give an affirmative answer to \textbf{Problem 1} for total collision solutions in the planar case. Then their result was generalized to partial collision solutions in the planar case by Gierzkiewicz, Schaefer and Zgliczynski \cite{GSZ24} and the authors \cite{WY25}. More recently Pinzari and Zgliczynski \cite{PZ26} studied total collision solutions in the spatial case, and Yu and Zhao \cite{YZ26} studied total collision solutions in the general $d$-dimensional cases after further assuming the dimension of the limiting CC is $d$ or $d-1$. 

These works essentially all followed the approach given by Moeckel and Montgomery \cite{MM25}, which is first to obtain a reduced system by modulo rotational symmetry, then use the \L ojasiewicz gradient inequality \cite{Lj82} to get a finite length theorem which guarantees the desired convergence. For this to work, one needs the assumption that the limiting set of $[\hat{q}^c_{\bk}(t)]$ is an isolated normalized CC to use the center-stable manifold theorem.

As is well known for the isolation assumption to hold for every normalized CC, it is equivalent to assume there are only finitely many normalized CCs up to rotational symmetry, which is a famous open problem in celestial mechanics and listed by Smale \cite{Sm98} as problem 6 in his problems for the 21st century. So far the problem is still widely open. It is only known to be true, when $n=3$ by the classical results of Euler and Lagrange, when $n=4$ by Hampton and Moeckel \cite{HM06}, and when $n=5$ for generic masses by Albouy and Kaloshin \cite{AK12}. For $n=6$, some partial progress has been made by Chang and Chen \cite{CC24}. 

In this paper, we shall prove the following result which says there is no infinite spin for any collision solution in all dimensions without the isolation assumption. 
\begin{theorem}
\label{thm;coll}  If $q(t)$ is a $\bk$-collision solution, the normalized relative configuration $\hat{q}^c_{\bk}(t)$ converges to a single CC in $\hat{\mathscr{C}}_{\bk}$, as $t \to T$ the collision moment. 
\end{theorem}
\begin{remark}
	When more than one partial collision occurs at the moment $T$, the above result holds for each partial collision. 
\end{remark}

Compared with previous work, the main difference is we study the original system directly instead of the reduced system, we study the original system directly. To make it work, we first need to generalize the classical \L ojasiewicz gradient inequality to a version that works in a small neighborhood of a compact set. Then we use it to prove a finite-length theorem for a particular time-dependent second-order gradient system with perturbation that fits the study of the collision solutions in the $n$-body problem (see Section \ref{sec:finite-length-analytic-gradient-equation}). This allows us to obtain the desired results in all dimensions without the isolation assumption.



Besides collision solutions, our approach can be applied to \emph{parabolic solutions} as well, which have a similar infinite spin problem. To explain it, we recall the definition of parabolic solutions. 
\begin{definition}\label{def;k-para}
\label{df;k-para} We say a solution $q: [t_0, \infty) \to \be$ is $\bk$-parabolic for some $\bk \subset \bn$ with $|\bk| \ge 2$, if there is a positive constant $C$, such that when $t \to \infty$, 
\begin{equation}
\label{eq;parabolic} 
\begin{cases}
  r_{ij}(t) \sim t^{\se}, & \; \forall i \ne j \in \bk; \\
r_{ij}(t) \ge C t, & \; \forall i \in \bk, \forall j \in \bk'
\end{cases} 
\end{equation}
Here and below $A\sim B$ means $A/B$ is bounded below and above by finite positive constants.  
\end{definition}

\begin{remark}
 We say it is partially parabolic, if $|\bk| < n$, and complete parabolic, if $|\bk|=n$. As we can see there may be more than one parabolic subsystem, when time goes to infinity.
\end{remark}

In dynamical systems, one of the main goals is to understand the long-term behavior of the system. For the $n$-body problem, the pioneering work was done by Chazy \cite{Chazy22} more than one hundred years ago with subsequent works in \cite{Pl67}, \cite{Sr71} and \cite{MS76}. In particular a complete classification of final motions for the three-body problem was given by Chazy in \cite{Chazy22}.

According to Chazy's classification, parabolic solutions (by the above definition, as time goes to infinity, for any two different masses from the $\bk$-subsystem, their relative distance goes to infinity and relative speed goes to zero) are critical cases that separate \emph{hyperbolic motions} (in this case the relative distance between two masses goes to infinity and relative speed goes to a non-zero limit, as time goes to infinity), and \emph{oscillatory motions} (in this case, the liminf of the relative distance between two masses is finite and the limsup is infinity, as time goes to infinity) or \emph{bounded motions} (in this case, the relative distance between two masses has a finite upper bound). 

It is interesting that for a $\bk$-parabolic solution $q(t)$, when the masses in the $\bk$-subsystem are far away from each other, they behave similarly as when they are close to a $\bk$-collision. In particular when $t \to \infty$, the normalized relative configuration curve $\hat{q}^c_{\bk}(t)$ must also approach the set of normalized central configurations $\hat{\mathscr{C}}_{\bk}$ (see Proposition \ref{prop;asym-para}). Because of this, a problem similar to \textbf{Problem 1} can be raised for parabolic solutions as well and there are also some failed attempts in the literature to solve it. In \cite{WY25}, we showed there is no infinite spin for parabolic solutions in the planar $n$-body problem under the isolation assumption. Here we prove there is no infinite spin for any parabolic solution in all dimensions without the isolation assumption. 

\begin{theorem}
\label{thm;para}  If $q(t)$ is a $\bk$-parabolic solution, the normalized relative configuration $\hat{q}^c_{\bk}(t)$ converges to a single CC in $\hat{\mathscr{C}}_{\bk}$, as $t \to \infty$. 
\end{theorem}
\begin{remark}
	When there is more than one parabolic subsystem when time goes to infinity, the above result applies to each such subsystem. 
\end{remark}

\section{Equations of a subsystem} \label{sec;co-tran}

Assume $q(t)$ is a solution of \eqref{eq;n-body} and choose an arbitrary $\bk \subset \bn$ with $k =|\bk| \ge 2$. In this section, we shall rewrite the equations of the subsystem $q_{\bk}$ as time-dependent differential equations. 

Let $c_{\bk}$ denote the center of mass of the $\bk$-subsystem and define the relative configuration as 
$$ z = (z_i)_{i \in \bk} \in \be_\bk, \quad z_i=q_i-c_\bk.$$
Introduce the mass metric $\| \cdot \|_{m}$ on $\be_{\bk}$ as induced by the following inner product 
$$\ip{v}{w}:=v^T M w,\ \ \ \ \ \ \forall v,w\in \be_\bk,$$
where $M={\rm diag}(m_i I_d)_{i \in \bk}$ is the $dk \times dk$ mass matrix. Then we can define
\begin{equation}
r=\norm{z},\quad s=\frac{z}{r}\in S_{\bk}. \label{eq;r-s}
\end{equation}

Introduce the following external force field  
\begin{align}\label{eq;external-filed}
    F_t(x)=\sum_{\ell\notin \bk}m_\ell\frac{q_\ell(t)-x}{|q_\ell(t)-x|^3}, \quad \forall x \in \rr^d \setminus \{ q_{\ell}(t): \ell \notin \bk \} 
\end{align}
and write the external acceleration on each $m_i$, $i \in \bk$ as
\begin{equation}
\label{eq;perturbation-term}
    f_i(t)= F_t(q_i(t)) - \ddot{c}_{\bk}=F_t(q_i(t))-m_{\bk}^{-1}\sum_{j\in \bk}m_j F_t(q_j(t))
\end{equation}
 
The motion of the $\bk$-subsystem now can be written as the following time-dependent system
\begin{align}\label{eq;eq-subsystem}
    \ddot{z}=\nabla_m U_{\bk}(z) + f_{\bk}(t), \quad f_{\bk}(t)=(f_i(t))_{i \in \bk}
\end{align} 
where $\nabla_m$ denotes the gradient with respect to the mass norm and   
$$U_\mathbf{k}(z)=\sum_{\substack{ \{ i < j\} \subset \bk}}\frac{m_im_j}{|z_i-z_j|} = U_{\bk}(q^c_{\bk}).$$ 


The next Lemma gives the uniform estimates of the perturbation term $f_{\bk}(t)$ that will be needed.
\begin{lemma}\label{lem;est-ft}
    If $q(t)$ is a $\bk$-collision solution on $[0, T)$, then for $t_0<T$ sufficiently close to $T$, there exists a constant $C>0$ such that
    \begin{align}\label{eq;est-ft-col}
        \norm{f_{\bk}(t)}\le C r(t), \qquad t_0\le t < T.
    \end{align}
    If $q(t)$ is a $\bk$-parabolic solution on $[0, \infty)$, then for sufficiently large $t_1$, there exists a constant $C>0$ such that
    \begin{align}
        \norm{f_{\bk}(t)}\le C r(t)t^{-3} = O(t^{-7/3}), \qquad t\ge t_1.
    \end{align}
\end{lemma}

\begin{proof}
    Let $G(y)=y/|y|^3$ be a vector-valued function and $y\in\rr^d$ is a column vector. The differentiation gives
$$
 DG(y)=|y|^{-3}I-3|y|^{-5}yy^{\mathsf T},
 \qquad \|DG(y)\|\le 2|y|^{-3}.
$$
Here the differential $DG$ can be seen as a matrix function, and $I$ is the identity matrix. 

For any $\{i \ne j\} \subset \bk$, by \eqref{eq;external-filed}, 
$$|F_t(q_i)-F_t(q_j)|\le \sum_{\ell\notin \bk}m_\ell|G(q_\ell-q_i)-G(q_\ell-q_{j})|.$$
In the collision case, since $q_i(t)-q_{j}(t)\to 0$ and $|q_\ell(t)-q_{j}(t)|$ is bounded away from zero, as $t\to T$, there exists $C_1>0$ and $t_0<T$, such that for all $\theta\in (0,1)$ and $t_0 \le t <T$, 
$$
|\theta(q_\ell-q_i)+(1-\theta)(q_\ell-q_{j})|\ge |q_\ell-q_{j}|-\theta|q_i-q_{j}|\ge C_1.
$$ 
Similarly in the parabolic case, by \eqref{eq;parabolic}, there exists $C_1>0$ and $t_1>0$, such that for all $\theta\in (0,1)$ and $t \ge t_1$,
$$
|\theta(q_\ell-q_i)+(1-\theta)(q_\ell-q_{j})|\ge C_1 t, 
$$
Then the mean-value estimate shows that there is a constant $C_2>0$, such that 
\begin{equation}
\label{eq;Fti-Ftl} |F_t(q_i)-F_t(q_j)|\le 
    \begin{cases}
    C_2|z_i-z_j|, & \quad t_0\le t < T \;  \text{ in the collision case},\\
 C_2t^{-3}|z_i-z_j|, &\quad t\ge t_1\;   \text{ in the parabolic case},
 \end{cases}
\end{equation}
By \eqref{eq;perturbation-term}, we may rewrite $f_i(t)$ as 
$$ f_i=m_{\bk}^{-1}\sum_{j \in \bk} m_j(F_t(q_i)-F_t(q_j)). $$
Plugging \eqref{eq;Fti-Ftl} into the above equation and using the fact 
\begin{equation}\label{eq;pair-identity}
 r^2=m_{\bk}^{-1}\sum_{ \{i <j\} \subset \bk} m_i m_j|z_i-z_j|^2
\end{equation}
there exists a constant $C>0$, such that
$$ \norm{f_{\bk}(t)}\le \begin{cases}
 C r(t), & \quad t_0\le t\le T \;  \text{ in the collision case},\\
  C r(t) t^{-3},  &\quad t\ge t_1\;   \text{in the parabolic case}.
\end{cases}
$$
The remaining estimate for the parabolic case follows from $r(t) \sim t^{2/3}$, when $t \ge t_1$, which is implied by \eqref{eq;parabolic} and \eqref{eq;pair-identity}. 
\end{proof}


\section{A finite-length theorem for an analytic second-order gradient system} \label{sec:finite-length-analytic-gradient-equation}

We first obtain a generalization of the classical \L  ojasiewicz gradient inequality \cite{Lj82}. Then use it to prove a finite-length theorem for a time-dependent second-order gradient system on a Riemannian manifold with perturbation suitable for the study of collision and parabolic solutions in the $n$-body problem. For a time-independent second-order gradient system in the Euclidean space without perturbation similar results and more have already been obtained in \cite{HJ98} and \cite{BBJ15}. 

Let $M$ be a finite-dimensional real analytic Riemannian manifold with the Riemannian inner product, norm, and distance denoted as $\langle\cdot,\cdot\rangle$, $\rnorm{\cdot}$, and $d$ respectively.
\begin{lemma}\label{lem;L-inq}
If $V:M\to\rr$ is a real analytic function and $\mathcal C$ is a nonempty compact subset of critical points of $V$ satisfying $V|_{\mathcal C}=V_0$ for some constant $V_0$, then there exist an open neighborhood $N$ of $\mathcal C$, a constant $C>0$, and an exponent $\theta\in(0,1/2]$ such that
\begin{equation}\label{eq;L-ineq}
 |V(y)-V_0|^{1-\theta}\le C\rnorm{\nabla V(y)},\qquad \forall y\in N.
\end{equation}
Here the gradient $\nabla$ is with respect to the Riemannian metric given on $M$. 
\end{lemma}

\begin{proof}
For each $x\in\mathcal C$, the classical \L{}ojasiewicz gradient inequality gives an open neighborhood $N_x$, a constant $C_x>0$, and an exponent $\theta_x\in(0,1/2]$ such that
\[
 |V(y)-V(x)|^{1-\theta_x}\le C_x\rnorm{\nabla V(y)},\qquad \forall y\in N_x.
\]
Here the usual inequality in analytic local coordinates yields the Riemannian version by equivalence of the corresponding norms on a smaller neighborhood. Since $V(x)=V_0$, we may shrink $N_x$ so that $|V(y)-V_0|\le1$ for all $y\in N_x$.

Since $\mathcal C$ is compact, we can find a finite subcover $N_{x_1},\ldots,N_{x_m}$ of $\mathcal C$. Set
\[
 \theta=\min_{1\le j\le m}\theta_{x_j},\qquad
 C=\max_{1\le j\le m}C_{x_j},\qquad
 N=\bigcup_{j=1}^m N_{x_j}.
\]
For any $y\in N$, choose $j$ with $y\in N_{x_j}$. Because $|V(y)-V_0|\le1$ and $1-\theta\ge1-\theta_{x_j}$, we obtain
\[
 |V(y)-V_0|^{1-\theta}
 \le |V(y)-V_0|^{1-\theta_{x_j}}
 \le C_{x_j}\rnorm{\nabla V(y)}
 \le C\rnorm{\nabla V(y)}.
\]
This proves \eqref{eq;L-ineq}.
\end{proof}

\begin{theorem}\label{thm;finite-length-analytic-gradient-equation}
Let $V$ and $\mathcal C$ be as in Lemma~\ref{lem;L-inq}, and $y(\tau) \in M$, $\tau \ge \tau_0$, be a $C^2$ solution of the following time-dependent second-order gradient system ($y'= dy/d\tau$)
\begin{equation} \label{eq;sec-order-equation}
D_\tau(y') + a(\tau) y' =\nabla V(y)+e(\tau) \quad ( D_\tau \text{ covariant differentiation})
\end{equation}
with time-dependent coefficient $a\in C^0(\rr,\rr)$ and perturbation $e\in C^0(\rr,TM)$. If
\begin{enumerate}[label=(\roman*)]
\item[$(i)$] $\dist(y(\tau),\mathcal C):=\inf_{x\in\mathcal C}d(y(\tau),x)\to0$, as $\tau\to\infty$;
\item[$(ii)$] $\rnorm{y'(\tau)}\to0$, as $\tau\to\infty$;
\item[$(iii)$] there exist constants $0<a_0\le a_1$ such that $a_0\le |a(\tau)|\le a_1$, when $\tau\ge\tau_0$; 
\item[$(iv)$] $e(\tau)\in T_{y(\tau)}M$ and $\rnorm{e(\tau)}\le \min\{1,c e^{-\kappa\tau}\},\ \tau\ge\tau_0$ for constants $c>0,\ \kappa>0$,
\end{enumerate}
then $y(\tau)$ converges to a point in $\mathcal C$, as $\tau \to \infty$, or more precisely
\begin{equation}\label{abstractlength}
 \int_{\tau_0}^{\infty}\bigl(\rnorm{y'(\tau)}+\rnorm{\nabla V(y(\tau))}\bigr)\,d\tau<\infty.
\end{equation} 
\end{theorem}

\begin{proof}
First choose an open neighborhood $N$ of $\mathcal C$ whose closure is compact and is contained in the neighborhood supplied by Lemma~\ref{lem;L-inq}. Thus \eqref{eq;L-ineq} holds throughout $N$. By assumption~$(i)$, $y(\tau)\in N$ for all sufficiently large $\tau$. Then continuity and compactness imply 
\[
 \rnorm{\nabla V(y(\tau))}\to0,\qquad V(y(\tau))\to V_0, \qquad \text{ as } \tau \to \infty.
\]
Choose a $\tau_0>0$ large enough, such that
\[ y(\tau)\in N, \qquad  \rnorm{\nabla V(y(\tau))}+\rnorm{y'(\tau)}\le1,\qquad \forall \tau\ge\tau_0.
\]
Since $\overline N$ is compact, we can find a constant $L\ge0$ such that
\begin{equation}
	\label{eq;Hess} \Big|\Hess V(y)[\xi,\xi] \Big|\le L \rnorm{\xi}^2,
 \qquad \forall y\in N \; \text{ and } \; \forall \xi\in T_yM.
\end{equation}

Now we will give a detailed proof when $a(\tau)$ is positive. For this let's introduce an energy function $E_{\ep}(\tau)$ as in \cite[(2.2)]{HJ98} which served as a Lyapunov function there
\begin{equation} \label{eq;E-ep}
	E_{\ep}= \tfrac12\rnorm{y'}^2-(V(y)-V_0)-\varepsilon\langle \nabla V(y), y'\rangle
\end{equation}
with $\varepsilon>0$ to be determined later. By \eqref{eq;sec-order-equation} and direct computations 
\begin{equation}
\label{Ederiv}
 \begin{aligned}
 	E_{\ep}'= & -a\rnorm{y'}^2-\varepsilon\rnorm{\nabla V(y)}^2
 -\varepsilon\Hess V[y',y'] \\
 & +\varepsilon a\langle \nabla V(y),y'\rangle +\langle e,y'\rangle-\varepsilon\langle \nabla V(y),e\rangle. 
 \end{aligned}
\end{equation}
For constants $a_0, a_1$ given in assumption $(iii)$, by Young's inequality 
\begin{align*}
 \varepsilon a_1\rnorm{\nabla V(y)}\rnorm{y'}
 &\le\tfrac{\varepsilon}{4}\rnorm{\nabla V(y)}^2+\varepsilon a_1^2\rnorm{y'}^2;\\
 \rnorm{e}\rnorm{y'}
 &\le\tfrac{a_0}{4}\rnorm{y'}^2+a_0^{-1}\rnorm{e}^2;\\
 \varepsilon\rnorm{\nabla V(y)}\rnorm{e}
 &\le\tfrac{\varepsilon}{4}\rnorm{\nabla V(y)}^2+\varepsilon\rnorm{e}^2.
\end{align*}
Substituting these estimates and \eqref{eq;Hess} into \eqref{Ederiv}, we obtain
\[
 E_{\ep}'\le
 \left(-\tfrac{3a_0}{4}+\varepsilon(L+a_1^2)\right)\rnorm{y'}^2
 -\tfrac{\varepsilon}{2}\rnorm{\nabla V(y)}^2
 +(a_0^{-1}+\varepsilon)\rnorm{e}^2.
\]
Choose an $\varepsilon>0$ small enough, such that $\varepsilon(L+a_1^2)\le a_0/4$. Assumption~$(iv)$ then implies
\begin{equation}\label{eq;inq-H}
 E_{\ep}'\le-c_0(\rnorm{y'}^2+\rnorm{\nabla V(y)}^2)+c_1e^{-2\kappa\tau},
\end{equation}
where $c_0=\min\{a_0/2,\varepsilon/2\}>0$ and $c_1=(a_0^{-1}+\varepsilon)c^2$. 

As the last term in \eqref{eq;inq-H} is positive, $E_{\ep}$ is not a Lyapunov function as in unperturbed systems considered in \cite{HJ98} and \cite{BBJ15}. To overcome this, we further modify the energy function as below
\[
 E=E_{\ep}+c_2 e^{-2\kappa\tau} \qquad \text{ with constant } c_2>(c_1+1)/(2\kappa). 
\] 
Let $c_3=\min\{c_0,2\kappa c_2-c_1\}>0$, then  \eqref{eq;inq-H} implies   
\begin{equation}\label{eq;Ecorrect}
 -E'\ge c_3(\rnorm{y'}^2+\rnorm{\nabla V(y)}^2+e^{-2\kappa\tau}).
\end{equation}
Notice that  $E(\tau)\to0$, when $\tau \to \infty$, as $y'\to0$, $\nabla V(y)\to0$, $V(y)\to V_0$. Then integrating \eqref{eq;Ecorrect} from large enough $\tau$ to $\infty$ shows  
\[
 E(\tau)\ge c_3\int_\tau^\infty e^{-2\kappa s}\,ds
 =\frac{c_3}{2\kappa}e^{-2\kappa\tau}>0.
\]
Thus positive fractional powers of $E(\tau)$ are well defined.

Let $p=1-\theta\in[1/2,1)$, where $\theta$ is from Lemma~\ref{lem;L-inq}. Then \L  ojasiewicz inequality implies 
\begin{equation}
\label{eq;Ep-bound}
\begin{aligned}
	 E^p & \le C_0\left(|V-V_0|^p+\rnorm{y'}^{2p}+\rnorm{\nabla V(y)}^p\rnorm{y'}^p+e^{-2p\kappa\tau}\right) \\
	 & \le C_1\left(\rnorm{\nabla V(y)}+\rnorm{y'}+e^{-\kappa\tau}\right). 
\end{aligned}
\end{equation}
Here $2p\ge1$ gives $\rnorm{y'}^{2p}\le\rnorm{y'}$ and $e^{-2p\kappa\tau} \le e^{-\kappa\tau}$, and the mixed term is obtained by 
$$ \rnorm{\nabla V(y)}^p\rnorm{y'}^p\le(\rnorm{\nabla V(y)}+\rnorm{y'})^{2p}\le \rnorm{\nabla V(y)}+\rnorm{y'}.$$

Combining \eqref{eq;Ecorrect}, \eqref{eq;Ep-bound} and the elementary inequality 
$$ A^2+B^2+C^2\ge(A+B+C)^2/3$$
we obtain the following inequality for a constant $c_4=\theta c_3/(3C_1)>0$,
\begin{equation}\label{lengthineq}
 -(E^\theta)'=- \theta E' E^{\theta -1}  \ge c_4\bigl(\rnorm{y'}+\rnorm{\nabla V(y)}+e^{- \kappa \tau}\bigr).
\end{equation}
Integration from $\tau_0$ to infinity then implies 
\begin{equation}\label{tailestimate}
 \int_{\tau_0}^\infty(\rnorm{y'}+\rnorm{\nabla V(y)}+ e^{- \kappa \tau})\,d\tau\le C_2 E(\tau_0)^\theta.
\end{equation}
This implies \eqref{abstractlength} and $y$ has finite length. Since $y(\tau)$ is contained in the compact set $\overline N$ for large $\tau$, it has a limit, which belongs to $\mathcal C$ by assumption $(i)$.

The proof for negative $a(\tau)$ is similar after replacing the definition of $E_{\ep}$ in \eqref{eq;E-ep} by
$$ E_{\ep}=-\tfrac12\rnorm{y'}^2+(V(y)-V_0)-\varepsilon\langle \nabla V(y), y'\rangle. $$

\end{proof}


\section{No infinite spin for collision solutions} \label{sec;coll-sol}

Let $q: [0, T) \to \be$ be a $\bk$-collision solution for the rest of the section with $|\bk| \ge 2$. If $z(t)$ is the relative configuration of the $\bk$-subsystem defined as in Section \ref{sec;co-tran}, then it satisfies equation \eqref{eq;eq-subsystem}. To better describe the asymptotic behavior of the solution near collision, following McGehee (see \cite{Mg74} or \cite{Mk89}), we introduce the new coordinates $(r, v, s, w)$ with $(r, s)$ defined as in \eqref{eq;r-s} and  
\begin{equation}\label{eq;McG-col}
  v=r^{1/2}\dot r,\qquad w=r^{3/2}\dot s.
\end{equation}
along with a new time parameter $\tau$ defined as $ dt = r^{3/2} d \tau$. Below we use $'$ to represent derivatives with respect to $\tau$.

Let $\gamma(\tau)=(r, v, s, w)(\tau)$ be the orbit of $(z,\dot{z})(t)$ in these new coordinates and time variable, and write the gradient of potential $U_{\bk}$ on sphere $S_{\bk}$ as
$$ \nabla_{S_{\bk}} U_{\bk}(s) = \nabla_m U_{\bk}(s) + U_{\bk}(s)s. $$   
\begin{proposition}\label{prop;MaC-col-eq}
$\gamma(\tau)$ satisfies the following time-dependent equations
\begin{equation}  \label{blowup-coll}
\begin{aligned}
 r'&=rv,\\
 v'&=\tfrac12v^2+\norm w^2-U_{\bk}(s)+ r^2\ip{s}{f_{\bk}}, \\
 s'&=w, \\
 w'&=\nabla_{S_{\bk}} U_{\bk}(s)-\tfrac12vw-\norm w^2s+r^2\bigl(f_{\bk}-\ip{s}{f_{\bk}}s\bigr), 
\end{aligned} 
\end{equation}
where the time dependence is through the term $f_{\bk}(\tau)$. Moreover the following identities hold
\begin{equation}\label{eq;cov-w}
     D_{\tau}w=\nabla_{S_{\bk}}U_{\bk}(s)-\tfrac12v(\tau)w + e(\tau), \quad \text{ with } \; e=r^2\bigl(f_{\bk}-\ip{s}{f_{\bk}}s\bigr);
\end{equation}
\begin{equation}\label{eq;scaled-energy-coll}
 rh_{\bk}=\tfrac12(v^2+\norm w^2)-U_{\bk}(s);
\end{equation}
\begin{equation}\label{eq;v-mono-coll}
 v'=\tfrac12\norm w^2+rh_{\bk}+r^2\ip{s}{f_{\bk}}.
\end{equation}
\end{proposition}

 \begin{proof}
 The first and third equations in \eqref{blowup-coll} follow directly from the definitions. 

Differentiate $r^2=\ip{z}{z}$ twice  
$$
\dot{r}^2+r\ddot{r}=\ip{\dot{r}s+r\dot{s}}{\dot{r}s+r\dot{s}}+\ip{rs}{\ddot{z}}.
$$
Combining $\ip{s}{\dot s}=0$ and \eqref{eq;eq-subsystem}, we have
\begin{equation}\label{eq;double-dot-r}
    \ddot r=r\norm{\dot s}^2-r^{-2}U_{\bk}(s)+\ip{s}{f_{\bk}}.
\end{equation}
Here we use the homogeneity of $U_{\bk}$ 
\begin{equation}\label{eq;homogeneity-U_k}
    \nabla_m U_{\bk}(rs)=r^{-2}\nabla_m U_{\bk}(s), \quad \ip{s}{\nabla_m U_{\bk}(s)}=-U_{\bk}(s).
\end{equation}
The second equation in \eqref{blowup-coll} then follows
$$v'=r^{3/2}(\frac{1}{2}r^{-1/2}\dot{r}^2+r^{1/2}\ddot{r})=\tfrac12v^2+\norm w^2-U_{\bk}(s)+r^2\ip{s}{f_{\bk}}.$$

Next we show \eqref{eq;cov-w} first. By differentiating $z=rs$ twice 
$$
\ddot{r}s+2\dot{r}\dot{s}+r\ddot{s}=\nabla_m U_{\bk}+f_{\bk}.
$$
With \eqref{eq;double-dot-r} and \eqref{eq;homogeneity-U_k}, this implies 
$$
\ddot{s}=r^{-3}\nabla_m U_{\bk}(s)+r^{-3}U_{\bk}(s)s+r^{-1}f_{\bk}-2r^{-3}vw-r^{-1}\ip{s}{f_{\bk}}s-\norm{\dot{s}}^2s.
$$

The tangential projection gives
$$
D_t\dot s=\ddot{s}-\ip{\ddot{s}}{s}s=r^{-3}\nabla_{S_{\bk}}U_{\bk}(s)+r^{-1}(f_{\bk}-\ip{s}{f_{\bk}}s)-2r^{-3}vw.
$$
Since $D_{t}w=r^{3/2}D_t \dot{s}+D_t(r^{3/2})\dot{s}$,
$$
 D_\tau w=\tfrac32vw+r^3D_t\dot s
 =-\tfrac12vw+\nabla_{S_{\bk}}U_{\bk}(s)+r^2\bigl(f_{\bk}-\ip{s}{f_{\bk}}s\bigr).
$$
This proves \eqref{eq;cov-w}. The last equation in \eqref{blowup-coll} then follows from the fact that 
$$ w' =  D_\tau w - \|w\|_m^2 s. $$ 

Finally, since $\dot z=r^{-1/2}(vs+w)$ and $\ip{s}{w}=0$, 
$$
h_{\bk}=\tfrac12\norm{\dot{z}}^2-U_{\bk}(rs)=r^{-1}\bigl(\tfrac12(v^2+\norm w^2)-U_{\bk}(s)\bigr),
$$
which proves \eqref{eq;scaled-energy-coll} and \eqref{eq;v-mono-coll}.
\end{proof}

Using Sundman's estimates (see \cite{Sp70} or \cite[Proposition 6.25]{FT04}):
\begin{align}\label{eq;asym-col}
    r(t) \sim |T-t|^\frac{2}{3}, \quad   \|\dot{z}(t)\|^2_m \sim |T-t|^{-\frac{2}{3}}, \; \text{ as } t \to T, 
\end{align}
we can show the energy of $\bk$-subsystem must be bounded.
\begin{lemma}\label{lem;est-energy-col}
    $h_{\bk}(t)$ is bounded.
\end{lemma}
\begin{proof}
    Recall that $h_\mathbf{k}=\frac{1}{2}\norm{\dot{z}}^2 -U_{\bk}(z)$. Differentiate both sides with respect to $t$ 
$$ \dot{h}_{\bk} =\ip{\ddot{z}}{\dot{z}}-\ip{\nabla_m U_{\bk}}{\dot{z}} =\ip{f_{\bk}}{\dot{z}}.$$
By Sundman's estimates in \eqref{eq;asym-col} and Lemma~\ref{lem;est-ft}, 
$$
\dot{h}_{\bk}=O(|T-t|^{1/3})
$$
which implies $h_{\bk}=O(1)$.
\end{proof}
The scaling $d\tau/dt=r^{-3/2}$ implies $\tau(t)\sim -\log(T-t)$. Hence $\tau \to \infty$, when $t \to T$. Moreover
$$  T-t(\tau)=O(e^{-\kappa\tau}), \; \text{ for some constant } \kappa>0. $$
Combining this with $\norm{f_{\bk}(t)}\le Cr(t)$ from Lemma~\ref{lem;est-ft}, we get
\begin{align}\label{eq;est-r2sf}
    r^2\ip{s}{f_{\bk}}=O(r^3)=O(e^{-2\kappa\tau}), \quad \text{ when } \tau \to \infty. 
\end{align} 

Now we can obtain the asymptotic behavior of $\gamma(\tau)$ as $\tau\to\infty$.
\begin{proposition}\label{prop;asym-col}
When $\tau \to \infty$, there exists a negative constant $v_0$ such that
$$
 v\to v_0,\quad w\to0,\quad
 U_{\bk}(s)\to  \tfrac12v_0^2,\quad
 \nabla_{S_{\bk}} U_{\bk}(s)\to0.
$$
In particular $\Om(s(\tau))$, the limiting set of $s(\tau)$, is a compact subset of $\hat{\scc}_{\bk}$ with $U_{\bk}|_{\Omega(s)}=v_0^2/2$.
\end{proposition}

\begin{proof}
By estimates \eqref{eq;asym-col} and Lemma \ref{lem;est-energy-col} 
$$ U_{\bk}(z(t))= \ey \|\dot{z}(t)\|^2_m-h_{\bk}(t)=O(|T-t|^{-2/3}). $$
As a result, when $\tau \to \infty$
    $$U_{\bk}(s)=rU_{\bk}(z)=O(1).$$
Hence $U_{\bk}(s(\tau))$ is bounded and the curve $s(\tau)$ remains in the compact sublevel set $\{ s \in \tilde{S}_{\bk}: U_{\bk}(s) \le C \}$, for some constant $C>0$. 

From Lemma~\ref{lem;est-energy-col}, $rh_{\bk}\to 0$ as $\tau\to \infty$ and since $U_{\bk}$ has a positive lower bound, \eqref{eq;scaled-energy-coll} implies that $|v|$ and $\norm{w}$ are bounded. Now integrate the equation \eqref{eq;v-mono-coll} and let
\begin{equation}
\label{eq;g-tau} g(\tau)=v(\tau)-\int_{\tau_0}^{\tau}rh_{\bk}+r^2\ip{s}{f_{\bk}}\,d \sg=v(\tau_0) + \tfrac12 \int_{\tau_0}^{\tau}\norm{w}^2 d \sg.
\end{equation}
This means $g(\tau)$ is bounded and nondecreasing, hence must have a finite limit. At the same time, by Lemma \ref{lem;est-energy-col} and the estimate \eqref{eq;est-r2sf}, the integral of $rh_{\bk}+r^2\ip{s}{f_\bk}$ is uniformly bounded and has a limit as $\tau\to\infty$. Therefore $v(\tau)$ has a limit $v_0$ as well.   

By the above estimates $\norm{w}$ must be bounded. Meanwhile as $s(\tau)$ stays in a compact sublevel set, $\nabla_{S_{\bk}} U_{\bk}(s(\tau))$ is bounded as well. The fourth equation in \eqref{blowup-coll} and 
$$ \| r^2(f_{\bk} - \ip{s}{f_{\bk}} s \|_m \le r^2 \|f_{\bk}\|_m = O(r^3)$$
imply that $w^\prime$ is bounded. Now we claim that $\norm{w}\to 0$, when $\tau \to \infty$. If not, there exists a constant $\ep_0>0$ and a sequence $\tau_n\to \infty$ with $\tau_{n+1}-\tau_n>1$ such that $\norm{w(\tau_n)}\ge \ep_0$. Meanwhile by the boundedness of $w'$, there is a constant $C>0$ such that for a fixed $\delta>0$, 
$$
\big|\norm{w(\tau_n+\delta)}-\norm{w(\tau_n)} \big|\le\norm{w(\tau_n+\delta)-w(\tau_n)}\le C\delta.
$$
Choose $\dl$ small enough that $C\delta<\ep_0/2$. Then $\norm{w(\tau)}>\ep_0/2$ for $\tau\in [\tau_n,\tau_n+\delta]$. As a result 
$$
\int_{\tau_1}^{\tau_{n+1}} \norm{w}^2 d \tau \ge  \frac{\delta  \ep^2_0}{4}n\to +\infty, \; \text{ as } n \to \infty. 
$$
which contradicts the convergence of this integral given above. 

We have proved that $(v, w)(\tau)$ converges to $(v_0, 0)$, as $\tau \to \infty$. By the energy identity \eqref{eq;scaled-energy-coll}, $U_{\bk}(s(\tau))\to \frac{1}{2}v_0^2$ and such a limit must be positive. Then $r' = r v$ and $r\to 0$ imply $v_0 <0$. 

It has been shown that $\gamma(\tau)$ converges to the set $\{r=0,v=v_0, U_{\bk}(s)=\frac{1}{2}v_0^2,w=0\}$ which is a compact subset of the collision manifold $\{r=0\}$. As \eqref{blowup-coll} is time-independent on the collision manifold $\{r=0\}$, the limit set $\Om(\gamma)$ is a nonempty compact subset invariant under the flow. This implies $w^\prime=0$ and $\nabla_{S_{\bk}} U_{\bk}(s(\tau))\to 0$, when $\tau \to \infty$.
\end{proof}

\begin{proposition}
\label{thm;no-spin-col}
There exists a normalized CC $s^*\in \hat{\scc}_{\bk}$ such that
\begin{equation}\label{eq;no-spin-col}
 \int_{t_0}^{T}\norm{\dot s(t)}\,dt<\infty \quad \text{ and } \quad s(t)\rightarrow s^*, \quad  \text{ as } t \rightarrow T. 
\end{equation} 
\end{proposition}

\begin{proof} 
Let $\Omega(s(\tau))$ denote the limit set of $s(\tau)$ obtained in Proposition~\ref{prop;asym-col}, which is a compact subset of critical points of $U_{\bk}$ contained in $\hat{\scc}_{\bk} \cap \{ U_{\bk} = v_0^2/2\}$. Let 
$$y=s, \quad V=U_{\bk}, \quad \mathcal C=\Omega(s(\tau)), \quad a=v/2. $$ 
By Proposition \ref{prop;MaC-col-eq}, $(y, y')(\tau)$ satisfies equations \eqref{eq;sec-order-equation} in Theorem~\ref{thm;finite-length-analytic-gradient-equation}. Meanwhile the estimate \eqref{eq;est-r2sf} and Proposition~\ref{prop;asym-col} imply the assumptions in Theorem \ref{thm;finite-length-analytic-gradient-equation} are satisfied, in particular
$$ \|e(\tau)\|_{m} = r^2\|f_{\bk}-\ip{s}{f_{\bk}}s\|_m \le r^2 \|f_{\bk}\|_m = O(r^3), \quad \tau \to \infty. $$
Then \eqref{eq;no-spin-col} can be obtained directly from Theorem~\ref{thm;finite-length-analytic-gradient-equation}.

\end{proof}

Theorem \ref{thm;coll} then follows directly from the above result.


\section{No infinite spin for parabolic solutions} \label{sec:para-sol}
Let $q: [0, \infty) \to \be$ be a $\bk$-parabolic solution for the rest of the section with $|\bk| \ge 2$. Like before $z(t)$ shall denote the motion of $\bk$-subsystem relative to $c_{\bk}$. Similar to the collision case, McGehee-type coordinates and time change can be introduced to study the asymptotic behavior of the masses close to infinity, such that the orbit $(z, \dot{z})(t)$ in the new coordinates and time is given by $\gamma(\tau)=(\rho,v,s,w)(\tau)$, where all the coordinates and time are defined as before except 
$$ \rho=r^{-1/2}. $$
 
\begin{proposition}\label{prop;MaC-para-eq}
$\gamma(\tau) = (\rho, v, s, w)(\tau)$ satisfies the following time-dependent equations
\begin{equation}  \label{blowup-para}
\begin{aligned}
\rho'&=-\frac{1}{2}\rho v,\\
 v'&=\tfrac12v^2+\norm w^2-U_{\bk}(s)+ \rho^{-4} \ip{s}{f_{\bk}}, \\
 s'&=w,\\
 w'&=\nabla_{S_{\bk}} U_{\bk}(s)-\tfrac12vw-\norm w^2s+\rho^{-4} \bigl(f_{\bk}-\ip{s}{f_{\bk}}s\bigr).
\end{aligned}
\end{equation}
where the time dependence is through the term $f_{\bk}(\tau)$. Moreover the following identities hold 
\begin{equation}
    D_{\tau}w=\nabla_{S_{\bk}}U_{\bk}(s)-\tfrac12v(\tau)w+e, \quad \text{ with } \; e=r^2\bigl(f_{\bk}-\ip{s}{f_{\bk}}s\bigr);
\end{equation}
\begin{equation}\label{eq;scaled-energy}
 h_{\bk}=\rho^2 \left( \tfrac12(v^2+\norm w^2)-U_{\bk}(s) \right);
\end{equation}
\begin{equation}\label{eq;v-mono}
 v'=\tfrac12\norm w^2+\rho^{-2}h_{\bk}+\rho^{-4} \ip{s}{f_{\bk}}.
\end{equation}
\end{proposition}
The result can be proved as in Proposition~\ref{prop;MaC-col-eq} and we omit the details here. 
The scaling $d\tau/dt=r^{-3/2}$ implies $\tau(t)\sim \log(t)$. Hence $\tau \to \infty$, when $t \to \infty$. It also implies 
$$  t^{-1}(\tau)=O(e^{-\kappa\tau}), \; \text{ for some constant } \kappa>0. $$

\begin{lemma}\label{lem;est-dot-z}
 $z(t)$ satisfies $ \dot{z}(t)=O(t^{-\frac{1}{3}})$, when $t\to \infty$.
\end{lemma}
\begin{proof}
	Recall that $\ddot{z}=\nabla_m U_{\bk}(z) + f_{\bk}(t)$. By Definition~\ref{df;k-para} and Lemma~\ref{lem;est-ft}, $\ddot{z}(t)=O(t^{-\frac{4}{3}})$. Then
    $$
    \dot{z}(t)=\dot{z}(t_0)+\int_{t_0}^t \ddot{z}(\sg) d \sg
    $$
    converges to a limiting vector $ A = (A_i)_{i \in \bk}$ and $\dot{z}(t)=A+O(t^{-\frac{1}{3}})$. By an integration
    $$ z(t) = At + O(t^{\se}). $$
    Since $\|z(t)\|_m = r(t) = O(t^{2/3})$,
    $$ z(t)/t = A + O(t^{-1/3}) \to 0, \quad t \to \infty. $$
    This implies $A=0$ and $ \dot{z}(t)=O(t^{-\frac{1}{3}})$. 
\end{proof}

\begin{lemma}\label{lem;est-u^-2h_k-e}
When $\tau \to \infty$,
$$ \rho^{-2}h_{\bk}(\tau)=O(e^{-\kappa \tau}), \quad e(\tau)=O(e^{-\kappa \tau})$$

\end{lemma}
\begin{proof}
Combining Definition~\ref{def;k-para} and Lemma~\ref{lem;est-ft}, we have 
$$
r(t) \sim t^{\se},\quad f_{\bk}(t)=O(t^{-\frac{7}{3}}), \quad t \to \infty
$$
and thus $e(t)=O(t^{-1})$. Recall that
$$
h_{\bk}=\frac{1}{2}\norm{\dot{z}}^2-U_{\bk}(z).
$$
By differentiating both sides, we get
$$
\dot{h}_{\bk}=\ip{\ddot{z}}{\dot{z}}-\ip{\nabla_m U_{\bk}}{\dot{z}}=\ip{f_{\bk}}{\dot{z}}.
$$
Thus $\dot{h}_{\bk}(t)=O(t^{-\frac{8}{3}})$. From Definition~\ref{def;k-para} and Lemma \ref{lem;est-dot-z}, $h_{\bk}(t) \to 0$, which implies 
\begin{equation}\label{eq;energy-para}
    h_{\bk}(t)=O(t^{-\frac{5}{3}})
\end{equation}
 and $\rho^{-2}(t)h_{\bk}(t)=O(t^{-1})$. By the relation between $t$ and $\tau$, we obtain the estimates 
 $$ \rho^{-2}h_{\bk}(\tau)=O(e^{-\kappa \tau}), \quad e(\tau)=O(e^{-\kappa \tau}). $$
\end{proof}

With the above estimates the next result follows from a similar argument as in the proof of Proposition \ref{prop;asym-col}. 
\begin{proposition}\label{prop;asym-para} When $\tau \to \infty$, there exists a positive constant $v_0$, such that
$$
 v\to v_0,\quad w\to0,\quad
 U_{\bk}(s)\to \tfrac12v_0^2,\quad
 \nabla_{S_{\bk}} U_{\bk}(s)\to0.
$$
In particular $\Om(s(\tau))$, the limiting set of $s(\tau)$, is a compact subset of $\hat{\scc}_{\bk}$ with $U_{\bk}|_{\Omega(s)}=v_0^2/2$. 
\end{proposition}

Theorem \ref{thm;para} follows directly from the next proposition, which itself can be proven by a similar argument as in the proof of Proposition \ref{thm;no-spin-col} using Theorem~\ref{thm;finite-length-analytic-gradient-equation} and Proposition~\ref{prop;asym-para}.

\begin{proposition}
\label{thm;no-spin-par}
There exists a normalized CC $s^*\in \hat{\scc}_{\bk}$ such that
\begin{equation*} 
 \int_{t_1}^{\infty}\norm{\dot s(t)}\,dt<\infty,
 \quad s(t)\rightarrow s^* \quad \text{ as } t \rightarrow \infty. 
\end{equation*}
 \end{proposition}

\hfill\newline
\bibliographystyle{abbrv}
\bibliography{Ref-Spin}

@article {HJ98,
    AUTHOR = {Haraux, A. and Jendoubi, M. A.},
     TITLE = {Convergence of solutions of second-order gradient-like systems
              with analytic nonlinearities},
   JOURNAL = {J. Differential Equations},
  FJOURNAL = {Journal of Differential Equations},
    VOLUME = {144},
      YEAR = {1998},
    NUMBER = {2},
     PAGES = {313--320},
      ISSN = {0022-0396,1090-2732},
   MRCLASS = {35L70 (35G20)},
  MRNUMBER = {1616968},
       DOI = {10.1006/jdeq.1997.3393},
       URL = {https://doi.org/10.1006/jdeq.1997.3393},
}

@article {BBJ15,
    AUTHOR = {B\'egout, Pascal and Bolte, J\'er\^ome and Jendoubi, Mohamed
              Ali},
     TITLE = {On damped second-order gradient systems},
   JOURNAL = {J. Differential Equations},
  FJOURNAL = {Journal of Differential Equations},
    VOLUME = {259},
      YEAR = {2015},
    NUMBER = {7},
     PAGES = {3115--3143},
      ISSN = {0022-0396,1090-2732},
   MRCLASS = {34D05 (03C64 35R20)},
  MRNUMBER = {3360667},
MRREVIEWER = {Jean-Philippe\ Rolin},
       DOI = {10.1016/j.jde.2015.04.016},
       URL = {https://doi.org/10.1016/j.jde.2015.04.016},
}

@article {WY25,
    AUTHOR = {Wang, Zhe and Yu, Guowei},
     TITLE = {THE PROBLEM OF INFINITE SPIN FOR PARABOLIC AND COLLISION 
SOLUTIONS IN THE PLANAR $n$-BODY PROBLEM},  
      YEAR = {arXiv:2507.05801, (2025)},
       
}

@article {PZ26,
    AUTHOR = {Pinzari, Gabriella and Zgliczynski, Piotr},
     TITLE = {No infinite spin for total collisions in the spatial N-body problem},  
      YEAR = {arXiv: 2604.22172, (2026)},
       
}

@article {YZ26,
    AUTHOR = {Yu, Xiang and Zhao, Lei},
     TITLE = {Exclusion of Infinite Spin for N-body problem in Rd},
      YEAR = {arXiv:2606.31196v1, (2026)},       
}

@article {AK12,
    AUTHOR = {Albouy, Alain and Kaloshin, Vadim},
     TITLE = {Finiteness of central configurations of five bodies in the
              plane},
   JOURNAL = {Ann. of Math. (2)},
  FJOURNAL = {Annals of Mathematics. Second Series},
    VOLUME = {176},
      YEAR = {2012},
    NUMBER = {1},
     PAGES = {535--588},
      ISSN = {0003-486X,1939-8980},
   MRCLASS = {70F10 (05C62 32C25 70G10)},
  MRNUMBER = {2925390},
MRREVIEWER = {Josep\ M.\ Cors},
       DOI = {10.4007/annals.2012.176.1.10},
       URL = {https://doi.org/10.4007/annals.2012.176.1.10},
}

@article {GSZ24,
    AUTHOR = {Gierzkiewicz, Anna and Schaefer, Rodrigo G. and Zgliczy\'nski,
              Piotr},
     TITLE = {No {I}nfinite {S}pin for {P}artial {C}ollisions {C}onverging
              to {I}solated {C}entral {C}onfigurations on the {P}lane},
   JOURNAL = {Comm. Math. Phys.},
  FJOURNAL = {Communications in Mathematical Physics},
    VOLUME = {406},
      YEAR = {2025},
    NUMBER = {7},
     PAGES = {Paper No. 158},
      ISSN = {0010-3616,1432-0916},
   MRCLASS = {70F10 (37C50 37D10)},
  MRNUMBER = {4915736},
       DOI = {10.1007/s00220-025-05340-3},
       URL = {https://doi.org/10.1007/s00220-025-05340-3},
}

@article {Chazy18,
    AUTHOR = {Chazy, Jean},
     TITLE = {Sur certaines trajectoires du problème des n corps},
   JOURNAL = {Bull. Astron.},
    VOLUME = {35},
      YEAR = {1918},
     PAGES = {321--389},
}

@article {Chazy22,
    AUTHOR = {Chazy, Jean},
     TITLE = {Sur l'allure du mouvement dans le probl\`eme des trois corps
              quand le temps cro\^{i}t ind\'{e}finiment},
   JOURNAL = {Ann. Sci. \'{E}cole Norm. Sup. (3)},
  FJOURNAL = {Annales Scientifiques de l'\'{E}cole Normale Sup\'{e}rieure. Troisi\`eme
              S\'{e}rie},
    VOLUME = {39},
      YEAR = {1922},
     PAGES = {29--130},
      ISSN = {0012-9593},
   MRCLASS = {DML},
  MRNUMBER = {1509241},
       URL = {http://www.numdam.org/item?id=ASENS_1922_3_39__29_0},
}

@article {CC24,
    AUTHOR = {Chang, Ke-Ming and Chen, Kuo-Chang},
     TITLE = {Toward finiteness of central configurations for the planar
              six-body problem by symbolic computations. ({I}) {D}etermine
              diagrams and orders},
   JOURNAL = {J. Symbolic Comput.},
  FJOURNAL = {Journal of Symbolic Computation},
    VOLUME = {123},
      YEAR = {2024},
     PAGES = {Paper No. 102277, 38},
      ISSN = {0747-7171,1095-855X},
   MRCLASS = {70F10 (68W30 70F15)},
  MRNUMBER = {4671620},
MRREVIEWER = {Jaime\ Burgos Garc\'ia},
       DOI = {10.1016/j.jsc.2023.102277},
       URL = {https://doi.org/10.1016/j.jsc.2023.102277},
}

@ARTICLE{FT04,
  author = {Ferrario, Davide L. and Terracini, Susanna},
  title = {{On the existence of collisionless equivariant minimizers for the
  classical {$n$}-body problem}},
  journal = {Invent. Math.},
  year = {2004},
  volume = {155},
  pages = {305--362},
  number = {2},
  coden = {INVMBH},
  doi = {10.1007/s00222-003-0322-7},
  fjournal = {Inventiones Mathematicae},
  issn = {0020-9910},
  mrclass = {70F10 (37J45 49J40 49S05 70F07 70F16 70H30)},
  mrnumber = {2031430 (2005b:70010)},
  mrreviewer = {Kuo-Chang Chen},
  url = {http://dx.doi.org.myaccess.library.utoronto.ca/10.1007/s00222-003-0322-7}
}

@article {HM06,
    AUTHOR = {Hampton, Marshall and Moeckel, Richard},
     TITLE = {Finiteness of relative equilibria of the four-body problem},
   JOURNAL = {Invent. Math.},
  FJOURNAL = {Inventiones Mathematicae},
    VOLUME = {163},
      YEAR = {2006},
    NUMBER = {2},
     PAGES = {289--312},
      ISSN = {0020-9910,1432-1297},
   MRCLASS = {70F15 (37N05 70F10)},
  MRNUMBER = {2207019},
MRREVIEWER = {Manuele\ Santoprete},
       DOI = {10.1007/s00222-005-0461-0},
       URL = {https://doi.org/10.1007/s00222-005-0461-0},
}

@incollection {Lj82,
    AUTHOR = {\L ojasiewicz, S.},
     TITLE = {Sur les trajectoires du gradient d'une fonction analytique},
 BOOKTITLE = {Geometry seminars, 1982--1983 ({B}ologna, 1982/1983)},
     PAGES = {115--117},
 PUBLISHER = {Univ. Stud. Bologna, Bologna},
      YEAR = {1984},
   MRCLASS = {58C05 (26E05 32C05 58C25)},
  MRNUMBER = {771152},
}

@article {Mg74,
    AUTHOR = {McGehee, Richard},
     TITLE = {Triple collision in the collinear three-body problem},
   JOURNAL = {Invent. Math.},
  FJOURNAL = {Inventiones Mathematicae},
    VOLUME = {27},
      YEAR = {1974},
     PAGES = {191--227},
      ISSN = {0020-9910,1432-1297},
   MRCLASS = {70.58 (58F99)},
  MRNUMBER = {359459},
MRREVIEWER = {J.\ Moser},
       DOI = {10.1007/BF01390175},
       URL = {https://doi.org/10.1007/BF01390175},
}

@article {Mk89,
    AUTHOR = {Moeckel, Richard},
     TITLE = {Chaotic dynamics near triple collision},
   JOURNAL = {Arch. Rational Mech. Anal.},
  FJOURNAL = {Archive for Rational Mechanics and Analysis},
    VOLUME = {107},
      YEAR = {1989},
    NUMBER = {1},
     PAGES = {37--69},
      ISSN = {0003-9527},
   MRCLASS = {58F25 (58F13 58F40 70F07 70F35)},
  MRNUMBER = {1000223},
MRREVIEWER = {Ernesto\ A.\ Lacomba},
       DOI = {10.1007/BF00251426},
       URL = {https://doi.org/10.1007/BF00251426},
}

@article {MM25,
    AUTHOR = {Moeckel, Richard and Montgomery, Richard},
     TITLE = {No infinite spin for planar total collision},
   JOURNAL = {J. Amer. Math. Soc.},
  FJOURNAL = {Journal of the American Mathematical Society},
    VOLUME = {38},
      YEAR = {2025},
    NUMBER = {1},
     PAGES = {225--241},
      ISSN = {0894-0347,1088-6834},
   MRCLASS = {70F10 (37N05 70F15 70F16 70G40)},
  MRNUMBER = {4810063},
       DOI = {10.1090/jams/1044},
       URL = {https://doi.org/10.1090/jams/1044},
}

@article {Pl67,
    AUTHOR = {Pollard, Harry},
     TITLE = {The behavior of gravitational systems},
   JOURNAL = {J. Math. Mech.},
  FJOURNAL = {J. Math. Mech.},
    VOLUME = {17},
      YEAR = {1967/68},
     PAGES = {601--611},
   MRCLASS = {70.34},
  MRNUMBER = {261826},
MRREVIEWER = {F.\ Nahon},
       DOI = {10.1512/iumj.1968.17.17036},
       URL = {https://doi.org/10.1512/iumj.1968.17.17036},
}

@article {Sr71,
    AUTHOR = {Saari, Donald G.},
     TITLE = {Expanding gravitational systems},
   JOURNAL = {Trans. Amer. Math. Soc.},
  FJOURNAL = {Transactions of the American Mathematical Society},
    VOLUME = {156},
      YEAR = {1971},
     PAGES = {219--240},
      ISSN = {0002-9947,1088-6850},
   MRCLASS = {70.34},
  MRNUMBER = {275729},
MRREVIEWER = {C.\ Marchal},
       DOI = {10.2307/1995609},
       URL = {https://doi.org/10.2307/1995609},
}

@article {MS76,
    AUTHOR = {Marchal, Christian and Saari, Donald G.},
     TITLE = {On the final evolution of the {$n$}-body problem},
   JOURNAL = {J. Differential Equations},
  FJOURNAL = {Journal of Differential Equations},
    VOLUME = {20},
      YEAR = {1976},
    NUMBER = {1},
     PAGES = {150--186},
      ISSN = {0022-0396,1090-2732},
   MRCLASS = {70.34 (85.34)},
  MRNUMBER = {416150},
       DOI = {10.1016/0022-0396(76)90101-7},
       URL = {https://doi.org/10.1016/0022-0396(76)90101-7},
}

@article {Sm98,
    AUTHOR = {Smale, Steve},
     TITLE = {Mathematical problems for the next century},
   JOURNAL = {Math. Intelligencer},
  FJOURNAL = {The Mathematical Intelligencer},
    VOLUME = {20},
      YEAR = {1998},
    NUMBER = {2},
     PAGES = {7--15},
      ISSN = {0343-6993,1866-7414},
   MRCLASS = {01A67 (00A05)},
  MRNUMBER = {1631413},
       DOI = {10.1007/BF03025291},
       URL = {https://doi.org/10.1007/BF03025291},
}

@article {Sp70,
    AUTHOR = {Sperling, Hans J.},
     TITLE = {On the real singularities of the {$N$}-body problem},
   JOURNAL = {J. Reine Angew. Math.},
  FJOURNAL = {Journal f\"ur die Reine und Angewandte Mathematik. [Crelle's
              Journal]},
    VOLUME = {245},
      YEAR = {1970},
     PAGES = {15--40},
      ISSN = {0075-4102,1435-5345},
   MRCLASS = {70.34},
  MRNUMBER = {290630},
MRREVIEWER = {F.\ V.\ Pohle},
       DOI = {10.1515/crll.1970.245.15},
       URL = {https://doi.org/10.1515/crll.1970.245.15},
}

\end{document}